\documentclass[preprint]{imsart}

\RequirePackage[OT1]{fontenc}
\RequirePackage{amsthm,amsmath,amssymb}
\RequirePackage[numbers]{natbib}
\RequirePackage[colorlinks,citecolor=blue,urlcolor=blue]{hyperref}
\RequirePackage{graphicx}
\usepackage{mathtools}

\startlocaldefs
\numberwithin{equation}{section}
\theoremstyle{plain}

\newtheorem{theorem}{Theorem}[section]
\newtheorem{lemma}{Lemma}[section]
\newtheorem{corollary}{Corollary}[section]

\newcommand{\E}{\mathbb{E}}
\newcommand{\R}{\mathbb{R}}

\DeclareMathOperator{\sgn}{sgn}
\DeclareMathOperator{\supp}{supp}

\DeclarePairedDelimiter\floor{\lfloor}{\rfloor}
\endlocaldefs

\begin{document}
\begin{frontmatter}
\title{The Accessible Lasso Models}
\runtitle{The Accessible Lasso Models}

\begin{aug}
\author{\fnms{Amir} \snm{Sepehri}\ead[label=e1]{asepehri@stanford.edu}}
\and
\author{\fnms{Naftali} \snm{Harris}\ead[label=e2]{naftali@stanford.edu}}

\address{Department of Statistics\\
Sequoia Hall\\
390 Serra Mall\\
Stanford University\\
Stanford, CA 94305-4065\\
\printead{e1,e2}}

\runauthor{Harris and Sepehri}
\affiliation{Stanford University}

\end{aug}

\begin{abstract}
A new line of research \citep{lee2013exact} on the lasso \citep{tibshirani1996regression} exploits the beautiful geometric fact that the lasso fit is the residual from projecting the response vector $y$ onto a certain convex polytope \citep{tibshirani2012degrees}. This geometric picture also allows an exact geometric description of the set of \textit{accessible} lasso models for a given design matrix, that is, which configurations of the signs of the coefficients it is possible to realize with some choice of $y$. In particular, the accessible lasso models are those that correspond to a face of the convex hull of all the feature vectors together with their negations. This convex hull representation then permits the enumeration and bounding of the number of accessible lasso models, which in turn provides a direct proof of model selection inconsistency when the size of the true model is greater than half the number of observations.
\end{abstract}



\tableofcontents
\end{frontmatter}

\section{Introduction}
The lasso \cite{tibshirani1996regression} has proved a popular approach in high dimensional regression problems, where the number of variables is large relative to the number of observations. Perhaps the major reason for this is that the estimated coefficients are sparse. Indeed, as noted first by \cite{osborne2000lasso}, the number of nonzero coefficients can never exceed the number of observations.

Let us consider a \textit{signed model} to be a subset of supported variables together with the signs of their corresponding coefficients. For a fixed design matrix, it is not in fact the case that any signed model with $n$ or fewer variables could possibly be chosen by the lasso for some choice of $y$. This differs from ordinary least squares, where all such models are possible.

In this paper, we will describe and count the signed models that are \textit{accessible}, that is, which can possibly be chosen by the lasso.

More formally, for a fixed design matrix $X_0 \in \R^{n \times p}$, the lasso solves the optimization problem
\begin{equation}\label{eq:lasso_regular}
\min_{\beta_0 \in \R^p} \frac{1}{2}||y - X_0\beta_0||_2^2 + \lambda ||\beta_0||_1
\end{equation}
for fixed $\lambda > 0$.

Under the relatively weak condition that $X_0$ is in general position, \cite{tibshirani2013lasso} showed that the minimizing $\beta_0$ is unique for any choice of $y \in \R^n$; he further showed that this condition holds almost surely, for example, if $X_0$ is chosen from any continuous probability distribution on $\R^{n \times p}$. For the remainder of this work, we will suppose that $X_0$ is such that $\beta_0$ is unique for any choice of $y$, and denote this by $\beta_0(y)$.

Since we are interested in the signs of $\beta_0$, it will be notationally more convenient to solve the following equivalent problem:
\begin{equation}\label{eq:lasso_expanded}
\begin{aligned}
& \min_{\beta \in \R^{2p}} & & \frac{1}{2}||y - X\beta||_2^2 + \lambda \sum_{j=1}^{2p}\beta_j \\
& \text{subject to} & & \beta \ge 0.
\end{aligned}
\end{equation}
Here, $X = \left(X_0,\; -X_0\right)$, and $\beta = \left(\beta_0^+,\; \beta_0^-\right)$. Note that the solution to \eqref{eq:lasso_expanded} will have at most one of $\beta_j$ and $\beta_{p+j}$ being nonzero. This is because if both are positive, one can subtract a common factor from both of them, reducing the sum of the $\beta$'s without changing the residual sum of squares.

We'll formalize the notion of a signed model by letting $\supp(\beta) = \{1 \le j \le 2p:\; \beta_j > 0\}$; of course, $\supp(\beta)$ is just an alternate description of the support and signs of $\beta_0$. For any signed model $S \subset \{1, \ldots, 2p\}$, let $X_S$ consist of just the columns of $X$ given by $S$. Similarly, let $1_S$ denote a vector of length $|S|$ with all ones.

Of great interest will be the values of $y$ leading to a particular signed model; define
\begin{equation}
A_S = \{y \in \R^n :\; \supp(\beta(y)) = S\}
\end{equation}

Clearly, $\bigcup_S A_S = \R^n$, forming a partitioning of the space into disjoint regions. We then call a signed model $S \subset \{1 \le j \le 2p\}$ \textit{accessible} if $A_S \ne \emptyset$; that is, if there is some value of $y$ such that $\supp(\beta(y)) = S$.

\section{Geometry of lasso model selection}
As we will see, the set $A_\emptyset$, or the ``null model polytope'', plays a special role in the geometry of the lasso. Another set which will play an equally important role is the convex hull of the columns of $X$, which we will denote by $\mathcal{CH}(X)$. In fact, $ A_\emptyset/\lambda$ and $\mathcal{CH}(X)$ are (by definition) polar duals, since $A_\emptyset = \{y \in \R^n:\; X^Ty \le \lambda 1\}$ by the KKT conditions for \eqref{eq:lasso_regular}.

Each face of $\mathcal{CH}(X)$ has a natural correspondence to a subset of signed variables $S$; namely, the signed variables $S$ that form the vertices of that face. We denote by $F_S$ the corresponding face of $A_\emptyset$, which is characterized by the property that $X_S^Tf = \lambda 1_S$ for any $f \in F_S$.

The sets $A_S$ are then characterized by the following lemma, which we believe was described and proved first by \cite{tibshirani2012degrees} in Lemma 3:
\begin{lemma}[Determining $\supp(\beta(y))$]\label{lemma:supp_beta}
Let $P_{A_\emptyset}(y)$ be the projection of $y$ onto the null model polytope $A_\emptyset$. Let $F_S$ be the face of $A_\emptyset$ of minimal dimension such that $P_{A_\emptyset}(y) \in F_S$. Then $y \in \overline{A_S}$. Furthermore, $y - P_{A_\emptyset}(y) = X_S \beta(y)$, the lasso fit.
\end{lemma}

In the special case where $X_0$ is orthogonal, it is well known that the lasso coefficients are equal to the $\lambda$-soft-thresholded inner products between $y$ and columns of $X$. When $p = n$, this orthogonal design case corresponds to $A_\emptyset$ being a hypercube with the origin as its center. (For $p < n$, it looks more like a ``cylinder'' built over a p-dimensional hypercube). Lemma \ref{lemma:supp_beta} is the natural extension of this soft-thresholding phenomenon to general design matrices; the regularization done by the lasso is to ``remove'' the projection onto $A_\emptyset$.

In fact, Lemma \ref{lemma:supp_beta} gives a complete characterization of the sets $A_S$, as illustrated in Figure \ref{fig:polytope}, taken from \cite{naftaliharris}.

\begin{corollary}[Geometry of Partitions $A_S$]\label{cor:asgeo}
If $A_S \ne \emptyset$, then it is the equal to $F_S + \{X_S\alpha:\; \alpha > 0\} = \{f + X_S\alpha:\; f \in F_S,\; \alpha > 0\}$.
\end{corollary}
\begin{proof}
Suppose $y \in A_S$. Then by Lemma \ref{lemma:supp_beta}, $y - P_{A_\emptyset}(y) = X_S \beta(y)$, so $y = {A_\emptyset}(y) + X_S\beta(y)$ $\in F_S + \{X_S\alpha:\; \alpha > 0\}$. Thus, $A_S \subset F_S + \{X_S\alpha:\; \alpha > 0\}$.

Since $\bigcup_S A_S = \R^n$, and $A_S \subset F_S + \{X_S\alpha:\; \alpha > 0\}$, we need only show that the sets $\left(F_S + \{X_S\alpha:\; \alpha > 0\}\right)_S$ are disjoint.

So consider $f_S + X_S \alpha_S$ and $f_T + X_T \alpha_T$, where $f_S \in F_S$ and $\alpha_S > 0$, (respectively for $T$). Then

\begin{align*}
& ||f_S + X_S \alpha_S - (f_T + X_T \alpha_T)||^2 \\
&= ||f_S - f_T||^2 + 2\langle f_S - f_T, X_S\alpha_S - X_T\alpha_T \rangle + ||X_S\alpha_S - X_T\alpha_T||^2\\
&> 2\langle f_S - f_T, X_S\alpha_S - X_T\alpha_T \rangle,
\end{align*}
where the strict inequality is from $X_S\alpha_S \ne X_T\alpha_T$ when $S \ne T$ and $\alpha_S, \; \alpha_T > 0$. Thus,

\begin{align*}
& ||f_S + X_S \alpha_S - (f_T + X_T \alpha_T)||^2 \\
&> 2\left(\alpha_S^T(\lambda 1_S - X_S^Tf_T) + \alpha_T^T(\lambda 1_T - X_T^Tf_S) \right) \\
&\ge 0,
\end{align*}
where $X_S^Tf_T \le \lambda 1_S$ and $X_T^Tf_S \le \lambda 1_T$ (coordinate-wise) by evaluating the KKT conditions at the points $f_T$ and $f_S$ in $A_\emptyset$.

Thus, $f_S + X_S \alpha_S \ne f_T + X_T \alpha_T$, so that $F_S + \{X_S\alpha:\; \alpha > 0\}$ and $F_T + \{X_T\alpha:\; \alpha > 0\}$ are disjoint.
\end{proof}

We remark that Corollary \ref{cor:asgeo} immediately shows that the sets $A_S$ are convex and have linear boundaries. It also shows that all nonempty sets $A_S$ are unbounded, with the possible exception of $A_\emptyset$.

\begin{figure}[ht]
    \includegraphics[width=\textwidth]{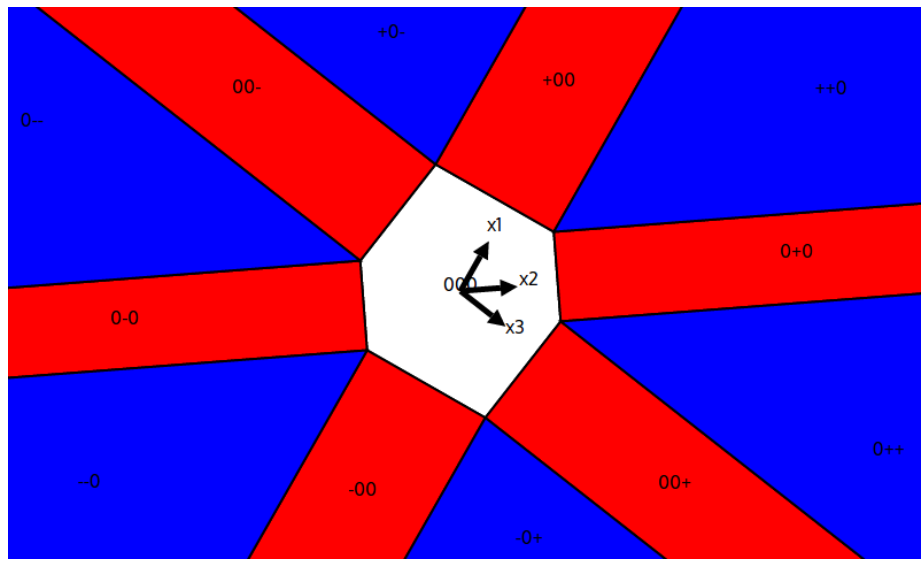}
    \caption{Example of sets $A_S \subset \R^2$ with $\lambda = 1.7$. $x_1$, $x_2$, and $x_3$ are the columns of $X_0 \in \R^{2 \times 3}$, each of which has length $1$. The middle white region is the polytope $A_\emptyset$, the red regions are the sets $A_S$ with $|S| = 1$, and the blue regions those with $|S| = 2$. Each region is labeled with $\sgn(\beta_0(y))$; for example, ``+0-'' means that the $x_1$ coefficient is positive, the $x_2$ coefficient is zero, and the $x_3$ coefficient is negative.}
    \label{fig:polytope}
\end{figure}

Now for any face $F_S$ of $A_\emptyset$, there is a point $f_S \in F_S$ such that $F_S$ is the lowest dimensional face containing $f_S$. Since $P_{A_\emptyset}(f_S + X_S\alpha) = f_S$ for any $\alpha > 0$, $S$ is in fact an accessible lasso model by Lemma \ref{lemma:supp_beta}. Thus, the set of all accessible models is determined by the faces of $A_\emptyset$, or alternatively by the faces of $\mathcal{CH}(X)$:

\begin{theorem}[Characterization of Accessible Models]\label{thm:hull_equivalence}
$A_S \ne \emptyset$ if and only if the variables $X_S$ form a face of $\mathcal{CH}(X)$.
\end{theorem}
\begin{proof}
From Lemma \ref{lemma:supp_beta}, $A_S \ne \emptyset$ if and only if the variables $X_S$ is the KKT-active set corresponding to the face $F_S$ of $A_\emptyset$. By the polar duality between $\lambda A_\emptyset$ and $\mathcal{CH}(X)$, this occurs if and only if the variables $X_S$ form a face of $\mathcal{CH}(X)$.
\end{proof}

We feel this is an intuitively believable result: The lasso attempts to construct a parsimonious approximation of $y$ using a positive linear combination of some of the columns of $X$. If it selected a subset of columns of $X$ that did not generate a face of $\mathcal{CH}(X)$, then intuitively it should be easier to generate a more parsimonious approximation using variables that do form a face of $\mathcal{CH}(X)$.

\section{Bounding the number of accessible lasso models}
Having established that the accessible models correspond to faces of $\mathcal{CH}(X)$, we can now apply classical results from polytope theory to count and bound the number of accessible models. The main result in this field is the celebrated upper bound theorem of \cite{mcmullen1970maximum}, which provides a tight upper bound on the number of faces of dimension $k$ of a polytope in $d$ dimensions.

To introduce the upper bound theorem, we will need the notion of a \textit{cyclic polytope}: A cyclic polytope with $v$ vertices in $d$ dimensions is the convex hull of $v$ points $\{(t_i, t_i^2, \ldots, t_i^d): i \in \{1, \ldots, v\}\}$, (and $t_i \ne t_j$ for any $i$, $j$).

\begin{theorem}[Upper Bound Theorem]\label{thm:upper_bound}
A convex polytope with $v$ vertices in $d$ dimensions has no more than $f_{k}(d, v)$ faces of dimension $k$, where $f_{k}(d, v)$ is the number of faces of dimension $k$ of a cyclic polytope with $v$ vertices in $d$ dimensions.

In particular, $f_{k}(d, v) = \binom{d}{k+1}$ for $0 \le k < \lfloor \frac{d}{2} \rfloor$, and in general is

\begin{equation}\label{eq:fk}
f_k(d, v) = \frac{v - \delta(v-k-2)}{v-k-1}\sum_{j=0}^{\lfloor d/2 \rfloor}\binom{v-1-j}{k+1-j}\binom{v-k-1}{2j-k-1+\delta},
\end{equation}
where $\delta$ is the parity of $d$.
\end{theorem}

The equation for $f_k(d, v)$ derives from the Dehn-Sommerville equations, and may be found, for example, in Theorem 3 of \cite{fukuda2000frequently}.

The upper bound theorem immediately allows us to bound the number of accessible lasso models of size $k$ (size of a model is the number of non-zero coefficients):
\begin{theorem}
The total number of accessible lasso models of size $k$ is no more than $f_{k-1}(n,2p)$.
\end{theorem}
\begin{proof}
By Theorem \ref{thm:hull_equivalence}, the accessible lasso models of size $k$ are exactly those whose associated feature vectors form a $(k-1)$-dimensional face of $\mathcal{CH}(X)$. $\mathcal{CH}(X)$ is a polytope in $n$ dimensions with at most $2p$ vertices, so by Theorem \ref{thm:upper_bound} it has no more than $f_{k-1}(n,2p)$ $(k-1)$-dimensional faces.
\end{proof}

By summing over McMullen's upper bound for faces of each dimension, we can also bound the total number of faces of a convex polytope of all dimensions, and hence, the total number of accessible lasso models:

\begin{lemma}\label{lem:total_faces}
For $v \ge 2d$, the total number of faces of a convex polytope with $v$ vertices in dimension $d$ is bounded by
\begin{equation}
2e(d+1)\left(\frac{2e(v-\lfloor d/2 \rfloor))}{\lceil d/2 \rceil}\right)^{\lfloor d/2 \rfloor}.
\end{equation}
\end{lemma}
\begin{proof}
Summing \eqref{eq:fk} over $k$ gives a tight upper bound on the total number of faces of all dimensions. Since the terms in the sum \eqref{eq:fk} are nonzero only if $j \ge k/2$, we may drop all other terms in that sum:

\begin{align*}
&\sum_{k=0}^{d-1}f_{k}(d, v) \\
&\le \sum_{k=0}^{d-1}\frac{v - \delta(v-k-2)}{v-k-1}\sum_{j \ge k/2}^{\lfloor d/2 \rfloor}\binom{v-1-j}{k+1-j}\binom{v-k-1}{2j-k-1+\delta} \\
&=\sum_{k=0}^{d-1}\frac{v - \delta(v-k-2)}{v-k-1}\sum_{j \ge k/2}^{\lfloor d/2 \rfloor}\binom{v-j}{v-k-1}\frac{v-k-1}{v-j}\binom{v-k-1}{v-2j-\delta} \\
&\le\sum_{k=0}^{d-1}\frac{v - \delta(v-k-2)}{v-\lfloor d/2 \rfloor}\sum_{j \ge k/2}^{\lfloor d/2 \rfloor}\binom{v-j}{v-k-1}\binom{v-k-1}{v-2j-\delta} \\
&\le\frac{v - \delta(v-d-1)}{v-\lfloor d/2 \rfloor}\sum_{k=0}^{d-1}\sum_{j \ge k/2}^{\lfloor d/2 \rfloor}\binom{v-j}{v-k-1}\binom{v-k-1}{v-2j-\delta} \\
&=\frac{v - \delta(v-d-1)}{v-\lfloor d/2 \rfloor}\sum_{k=0}^{d-1}\sum_{j \ge k/2}^{\lfloor d/2 \rfloor}\binom{v-j}{v-2j-\delta}\binom{j+\delta}{k-j+1}, \\
\end{align*}
using the identity $\binom{r}{s}\binom{s}{t} = \binom{r}{t}\binom{r-t}{r-s}$. Exchanging the order of summation gives
\begin{align*}
\ldots&=\frac{v - \delta(v-d-1)}{v-\lfloor d/2 \rfloor}\sum_{j=0}^{\lfloor d/2 \rfloor}\binom{v-j}{v-2j-\delta}\sum_{k=0}^{2j}\binom{j+\delta}{k-j+1}, \\
&=\frac{v - \delta(v-d-1)}{v-\lfloor d/2 \rfloor}\sum_{j=0}^{\lfloor d/2 \rfloor}\binom{v-j}{v-2j-\delta}2^{j+\delta}.
\end{align*}
For $v \ge 2d$, the terms in the sum are nondecreasing in $j$, so we may bound the sum by
\begin{align*}
\ldots&\le\frac{v - \delta(v-d-1)}{v-\lfloor d/2 \rfloor}\lfloor d/2 \rfloor \binom{v-\lfloor d/2 \rfloor}{v-d}2^{\lceil d/2 \rceil} \\
&\le\frac{v - \delta(v-d-1)}{v-\lfloor d/2 \rfloor}\lfloor d/2 \rfloor \left(\frac{(v-\lfloor d/2 \rfloor)e}{\lceil d/2 \rceil}\right)^{\lceil d/2 \rceil} 2^{\lceil d/2 \rceil} \\
&\le 2e(d+1)\left(\frac{2e(v-\lfloor d/2 \rfloor))}{\lceil d/2 \rceil}\right)^{\lfloor d/2 \rfloor}, \\
\end{align*}
using the fact that $\binom{n}{n-k} \le \left(\frac{ne}{k}\right)^k$ in the first inequality and $v < 2(v-d/2)$ in the last inequality, (for even $d$).
\end{proof}

Now, the number of accessible lasso models is exactly equal to the number of non-empty faces of $\mathcal{CH}(X)$, a convex polytope in $n$ dimensions with at most $2p$ vertices. Thus, the following corollary is immediate:
\begin{corollary}
For $p \ge n$, the total number of accessible lasso models is bounded by
\begin{equation*}
2e(n+1)\left(\frac{2e(2p-\lfloor n/2 \rfloor))}{\lceil n/2 \rceil}\right)^{\lfloor n/2 \rfloor}.
\end{equation*}
\end{corollary}

In the setting where $p \sim \rho n$, this bound is
\begin{equation*}
O\left(n\left(4e(2\rho - \frac{1}{2})\right)^{n/2}\right).
\end{equation*}

Naively, for $n \le p$, one would bound the total number of lasso models by $\sum_{k = 0}^{n}2^k \binom{p}{k}$. In the setting $p \sim \rho n$, this is at least
\begin{align*}
2^n \binom{p}{n} &\approx 2^n \left(\frac{pe}{n}\right)^n \\
                 &\approx \left(2 \rho e\right)^n.
\end{align*}

Our bound gives $\left(\sqrt{4e(2\rho - \frac{1}{2})}\right)^n < \left(\sqrt{8\rho e}\right)^n \ll \left(2 \rho e\right)^n$, cutting down the number of models exponentially from the naive bound.

$\mathcal{CH}(X)$ is no ordinary convex polytope; it also exhibits symmetry through the origin. One might expect this special structure to improve the bound on the number of faces. However, results of \cite{barvinok2009centrally} provide evidence suggesting that dramatic improvements to the upper bound theorem for centrally symmetric polytopes are unlikely.

\section{Model selection inconsistency}
The bound on the number of accessible lasso models implies that correctly selecting large true models is impossible: Suppose we have a sequence of lasso problems of the form \eqref{eq:lasso_regular}, with $n \to \infty$, and $p_n/n \to \rho$. Suppose that for each $n$, there is a ``true model'' $S_n$, of size $k_n = |S_n|$, with $k_n/n \to \kappa$.

\begin{theorem}\label{thm:inconsistency}
Suppose, for each $n$, that the entries of $X_0$ are distributed iid according to some symmetric distribution on $\R$. Then if $\kappa > \frac{1}{2}$, for any $\rho > \frac{\kappa}{\kappa - \frac{1}{2}}$ we have
$P(A_{S_n} = \emptyset) \to 1$.
\end{theorem}
\begin{proof}
Since the entries of $X_0$ are iid and symmetric, all the signed models of size $k$ are equally likely to be a face of $\mathcal{CH}(X)$. Thus, the probability that $A_{S_n} \ne \emptyset$ depends only on $k = |S_n|$. This probability is
\begin{align}
P(A_{S_n} \ne \emptyset) &= \frac{\E\left[\text{\# accessible models of size k}\right]}{2^k \binom{p}{k}}\notag \\
                     &\le \frac{f_k(n, 2p)}{2^k \binom{p}{k}}\notag\\
                     &\le \frac{\frac{2p}{2p-k-1}\sum_{j=0}^{\lfloor n/2 \rfloor}\binom{2p-1-j}{k+1-j}\binom{2p - k - 1}{2j - k}}{2^k \binom{p}{k}} \label{eq:sum_bound}
\end{align}
First of all, observe that the terms in the sum are zero unless $j \ge \frac{k}{2}$. Next, we'll show that that $\binom{2p-1-j}{k+1-j}\binom{2p-1-j}{k+1-j}$ is increasing in $j$, and therefore maximized at $j = \lfloor\frac{n}{2}\rfloor$. To see this, define $a_j = \binom{2p-1-j}{k+1-j}\binom{2p-1-j}{k+1-j}$, and consider
\begin{align*}
\frac{a_{j+1}}{a_{j}} & = \frac{(k-j+1)(2p-2j-1)(2p-2j-2)}{(2p-j-1)(2j-k+2)(2j-k+1)} \\
                                  &\ge \frac{(k-n/2+1)(2p-n-1)(2p-n-2)}{(2p-k/2-1)(n-k+2)(n-k+1)}\\
                                  &=\frac{(k_n/n-1/2+1/n)(2p_n/n-1-1/n)(2p_n/n-1-2/n)}{(2p_n/n-k_n/(2n)-1/n)(1-k_n/n+2/n)(1-k_n/n +1/n)}\\
                                   &\ge \frac{(\kappa-1/2-\epsilon)(2\rho -1 -\epsilon)^2}{(2\rho -\kappa /2 + \epsilon)(1-\kappa + \epsilon)^2} \\
                                   &\doteq K(\kappa,\rho, \epsilon).
\end{align*}
If we have $K(\kappa,\rho, 0) > 1$, then we may choose $\epsilon$ small enough such that $K(\kappa,\rho, \epsilon) > 1$ as well, and then the term are increasing in $j$ for sufficiently large $n$. In fact, a sufficient condition for $K(\kappa, \rho, 0) > 1$ is that $\rho > \frac{\kappa}{\kappa - \frac{1}{2}}$, so the terms are in fact increasing.

We can thus bound each of the terms in the numerator of \eqref{eq:sum_bound} by the maximum, yielding the bound
\begin{equation}\label{eq:sum_ub}
\left(\frac{n-k}{2}\right)\binom{2p-1-\frac{n}{2}}{k+1-\frac{n}{2}}\binom{2p-k-1}{n-k}
\end{equation}
for the sum in the numerator of \eqref{eq:sum_bound}. Naturally, we substitute the bound \eqref{eq:sum_ub} in for \eqref{eq:sum_bound}, yielding:

\begin{align}
P(A_{S_n} \ne \emptyset) &\le \frac{\frac{2p}{2p-k-1}\left(\frac{n-k}{2}\right)\binom{2p-1-\frac{n}{2}}{k+1-\frac{n}{2}}\binom{2p-k-1}{n-k}}{2^k \binom{p}{k}}\notag \\
&\le \frac{\frac{2\rho + \epsilon}{2\rho - \kappa - \epsilon}\left(\frac{n(1 - \kappa + \epsilon)}{2}\right)\binom{(2\rho-\frac{1}{2} + \epsilon)n}{(\kappa - \frac{1}{2} + \epsilon)n}\binom{(2\rho-\kappa + \epsilon)n}{(1-\kappa + \epsilon)n}}{2^{(\kappa - \epsilon) n} \binom{(\rho - \epsilon) n}{(\kappa - \epsilon) n}}\label{eq:intermediate},
\end{align}
for any choice of $\epsilon > 0$ and sufficiently large $n$.

The expression above involves terms of the form $\binom{an}{bn}$, which we may approximate using Stirling's formula as follows:
\begin{align}
\binom{an}{bn} &=\left(1 + o(1)\right) \frac{\sqrt{2 \pi an}e^{-an}(an)^{an}}{\sqrt{2 \pi bn}e^{-bn}(bn)^{bn}\sqrt{2 \pi (a-b)n}e^{-(a-b)n}((a-b)n)^{(a-b)n}}\notag \\
               &=\left(1 + o(1)\right) \sqrt{\frac{a}{2\pi b(a-b) n}} \left(\frac{a^a}{b^b (a-b)^{a-b}}\right)^n \label{eq:uncle_stirling}
\end{align}

Substituting the approximation \eqref{eq:uncle_stirling} into each of the three binomial terms in \eqref{eq:intermediate} yields:
\begin{align*}
P(A_S \ne \emptyset) &\le O(\sqrt{n})\frac{\left(\frac{(2\rho-\frac{1}{2}+\epsilon)^{2\rho-\frac{1}{2}+\epsilon}}{(\kappa-\frac{1}{2}+\epsilon)^{\kappa-\frac{1}{2}+\epsilon}(2\rho - \kappa)^{2\rho - \kappa}}\right)^n
\left(\frac{(2\rho-\kappa+\epsilon)^{2\rho-\kappa+\epsilon}}{(1-\kappa+\epsilon)^{1-\kappa+\epsilon}(2\rho - 1)^{2\rho - 1}}\right)^n
}{
\left(2^{\kappa-\epsilon}\frac{(\rho-\epsilon)^{\rho-\epsilon}}{(\kappa-\epsilon)^{\kappa-\epsilon} (\rho - \kappa)^{\rho - \kappa}}\right)^n
}\\
&\doteq O(\sqrt{n})C(\rho, \kappa, \epsilon)^n \\
\end{align*}

Thus, $P(A_S \ne \emptyset) \to 0$ if $C(\rho, \kappa, \epsilon) < 1$. We may choose $\epsilon$ arbitrarily small, and so in fact the same holds if 

\begin{equation*}
C(\rho, \kappa, 0) =
\left(\frac{(2\rho-\frac{1}{2})^{2\rho-\frac{1}{2}}(\frac{\kappa}{2})^{\kappa}(\rho - \kappa)^{\rho - \kappa}
}{
(\kappa-\frac{1}{2})^{\kappa-\frac{1}{2}}(1-\kappa)^{1-\kappa}(2\rho-1)^{2\rho-1}\rho^\rho
} \right)^n < 1.
\end{equation*}

For any $\kappa > 1/2$ there exists a $\rho_0$ such that for $\rho > \rho_0$ we have $C(\rho, \kappa) < 1$. Concretely, $\rho_0 = \frac{\kappa}{\kappa-\frac{1}{2}}$ satisfies this.
\end{proof}

Of course, one might consider it overly stringent to require that the specific true model $S$ be accessible; one might be happy even if some ``similar'' model to the true model were accessible. However, Theorem \ref{thm:inconsistency} combined with a simple union bound shows that this is impossible even for reasonably close models.

\begin{corollary}\label{cor:eps_ball}
Consider the same sequence of lasso problems as in Theorem \ref{thm:inconsistency}. Let $B_\epsilon(S) = \{T \subset \{1, \ldots, 2p\}:\; \left|S \triangle T\right| < \epsilon n\}$. Then for $\rho$ and $\kappa$ such that $\rho > \frac{\kappa}{\kappa - \frac{1}{2}}$, there exists $\epsilon = \epsilon(\rho, \kappa) > 0$ such that
\begin{equation*}
P\left(\left(\bigcup_{T \in B_{\frac{\epsilon}{\log{n}}}(S_n)} A_T\right) \ne \emptyset\right) \to 0.
\end{equation*}
\end{corollary}

\begin{proof}
We have that 
\begin{align*}
|B_{\epsilon/\log n}(S_n)| &\le (2p)^{\frac{\epsilon n}{\log(n)}} \\
&= (2\rho n)^{\frac{\epsilon n}{\log(n)}} \\
&= (e^{\log(2\rho n)})^{\frac{\epsilon n}{\log(n)}} \\ 
&= e^{\epsilon \log(2\rho)\frac{n}{\log(n)}}e^{\epsilon n} \\
&\le e^{2\epsilon n}.
\end{align*}

Note that $C(\rho, \kappa, 0)$ is a continuous function, hence for $\rho$ and $\kappa$ such that $C(\rho, \kappa, 0) < 1$, we can choose $n$ large enough such that perturbing $\kappa$ by $\frac{\epsilon}{\log(n)}$ to $\kappa^\prime$ keeps $C(\rho, \kappa^\prime, 0)$ less than $1-\delta$, for small enough $\delta$. Then by a simple union bound we get
\begin{align*}
P\left(\bigcup_{T \in B_{\epsilon / \log n}(S_n)} A_T \ne \emptyset\right) &\le  O(\sqrt{n})e^{2\epsilon n} (1-\delta)^n
\end{align*}
Thus we can choose $\epsilon$ small enough such that $e^{2\epsilon}(1-\delta) <1$, which makes the right hand side go to zero as $n$ goes to infinity.
\end{proof}
\textit{Remark.} It is important to note the difference between a model being accessibe and the same model being selected by lasso. Accessibility is a property of the model and the design matrix. On the other hand, for a model being selected not only do we need it to be accessible but we also need the response vector to lie in a certain subset of the space. This means that being selected is more stringent than accessibility. Our argument for model selection inconsistency is based on non-accessibility of most of large models, and for that reason it is a sub-optimal result. In fact, Corollary 7.1. in \cite{donoho2009counting} implies that we can not improve drastically upon Theorem \ref{thm:inconsistency} to relax the assumption $\kappa >1/2$. Using notation, and in the context, of Theorem \ref{thm:inconsistency} it asserts that we can not assume a lower bound  than $(2e\log(\rho))^{-1}$ on $\kappa$ smaller, when entries of the design matrix are independent Gaussian variables.
\section{Simulations}
This sections shows results from simulation studies of the model selection consistency for lasso. We follow the setup from the previous section assuming that $p_n/n \rightarrow \rho > 1$, and $k_n/n \rightarrow \kappa$. We consider random design matrix with iid entries as well as independent observations from correlated features.

We consider the normalized Hamming distance between the fitted signed model and the true signed model as a measure of mis-selection. In particular, \textit{relative selection error} is defined as $\frac{\|\beta-\hat{\beta}\|_0}{\|\beta\|_0}$, where $\|.\|_0$ is the number of non-zero elements and $\beta$ is a signed model compatible with the notation of \ref{eq:lasso_expanded}. Figure \ref{fig:path} illustrates the selection error along a lasso path.
\begin{figure}[ht]
    \includegraphics[width=\textwidth]{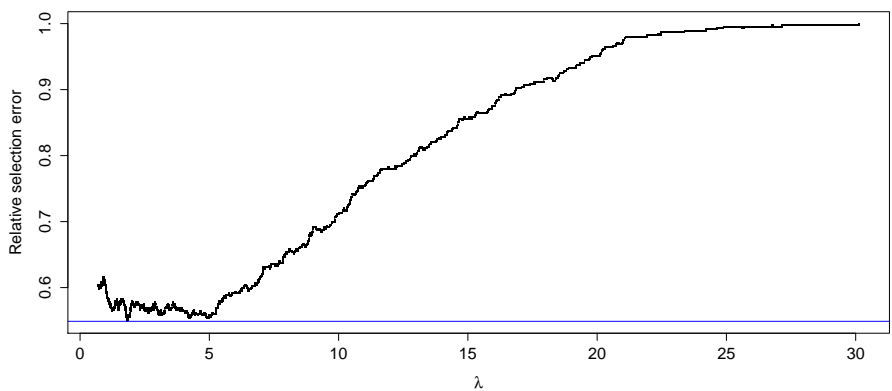}
    \caption{Example of relative selection error along a lasso path. Lasso path is generated using the R package `glmnet', on $n = 1000$ observations of $p = 1200$ variables on a net of $10p$ values of $\lambda$. The design matrix has iid Gaussian entries. The coefficients vector $\beta$ has 450 randomly selected entries equal to 10, and the rest of entries equal to zero. Noise is independent of the design matrix and has standard normal distribution.The blue line shows the minimum selection error achieved by a lasso selected model.}
    \label{fig:path}
\end{figure}

First, we examine the results in a setup where the conditions for our theorems hold. Assume we have $n =100$ observations of $p = 120$ variables, and a response variable. The coefficients vector $\beta$ has $k$ randomly selected entries equal to 10, and the rest of entries equal to zero. $k$ takes values $5,10,20,30,45,60$. For each instance of the model we search over the entire lasso path to find the closest possible lasso model to the true model. We then report the normalized hamming distance between the best lasso model and the true model as the relative selection error. That is $\frac{\|\beta-\hat{\beta}\|_0}{k}$, where $\hat{\beta}$ is the closest lasso model and $\beta$ is the true model. For each $n,p$ and, $k$ we repeat this procedure 1000 time and show the histogram of the relative selection error of the best models selected by lasso. Figure \ref{fig:histbasic} shows these histograms. As it can be seen, for small $k$, i.e. a sparser model, lasso can select a model pretty close to the true model. As $k$ grows, the selection becomes worse. In particular, for $k\sim n/2$ lasso doesn't select a model even close to the true model.

We also try a similar simulation when the assumption of iid entries for the design matrix is relaxed to a more realistic condition. We assume having independent observations of correlated features. Assume $n=100$, $p=120$, and $k= 45$. The design matrix consists of $n$ observations from a $p-$variate normal distribution with mean zero, variance one, and all the pairwise correlations equal to $\rho$. In this study $\rho$ takes on values $0,0.1, 0.3, 0.5, 0.7 ,0.9$. The rest of the setup is similar to that of the previous paragraph. This is illustrated in Figure \ref{fig:histCorr}. It suggests that the selection error becomes slightly worse when the features are correlated.

To summarize, simulation studies are consistent with the theoretical results from previous section. They also suggest that the results of this paper, although stated in an asymptotic
regime, are in agreement with simulations for relatively small values of $n$. We conclude this section with the following remark.

\textit{Remark.} It is worth mentioning that ideally one would try to numerically investigate accessibility of lasso models. This is computationally demanding, if not impossible, because it requires finding the convex hull of feature vectors in $n$ dimensions. State of the art methods for computing the convex hull of $2p$ points in $n$ dimensions require $O((2p)^{\floor*{n/2}})$ operations, which is prohibitive for relatively small $n$ and impossible for real problems of interest. For this reason, we focus on lasso selected models for simulation purposes.
\begin{figure}[h]
    \includegraphics[width=\textwidth]{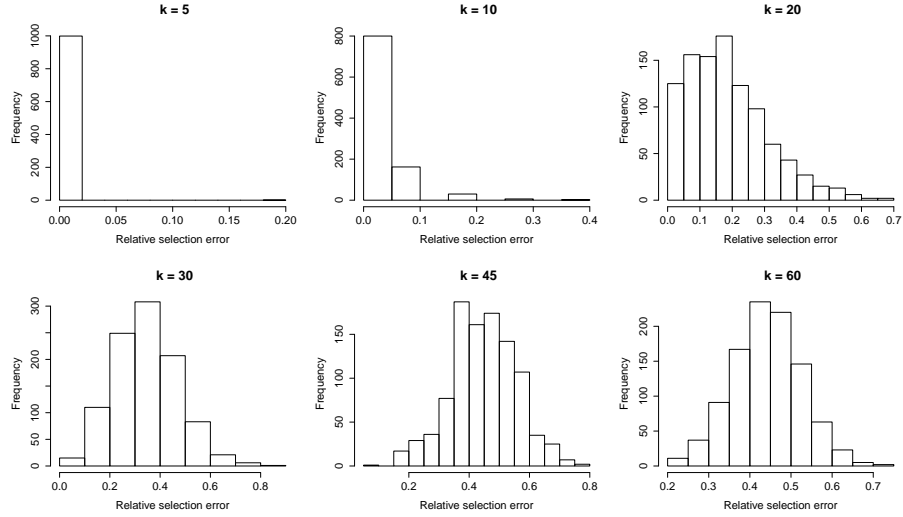}
    \caption{Histogram of relative selection error for $n= 100$, $p=120$, and different sparsity level $k$. Each histogram is generated using 1000 repetitions.}
    \label{fig:histbasic}
\end{figure}

\begin{figure}[h!]
    \includegraphics[width=\textwidth]{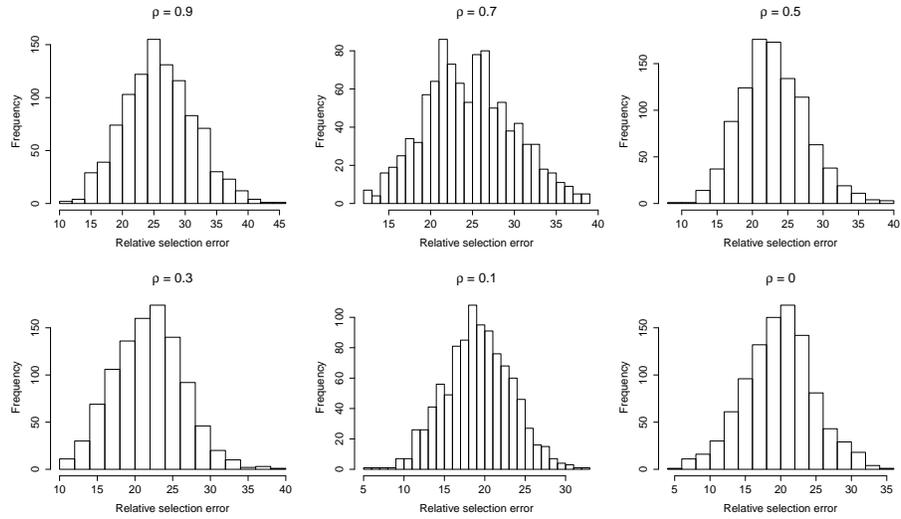}
    \caption{Histogram of relative selection error for $n= 100$, $p=120$, $k=45$, and different correlation parameters $\rho$. Each histogram is generated using 1000 repetitions.}
    \label{fig:histCorr}
\end{figure}

\section{Discussion}
We have characterized the geometry of the partitioning of the space $\mathbb{R}^n$ of the response vector into the regions $\{A_S\}_S$ corresponding to different accessible signed models $S$. We have also showed that these accessible models correspond to faces of the convex hull of the design matrix and its negative counterpart, $\mathcal{CH}(X)$. We then used the upper bound theorem from polytope combinatorics to bound the number of accessible lasso models of each size, and then used this bound to directly prove a model selection inconsistency result.

It is worth mentioning that all of our results hold regardless of the choice of the regularization parameter $\lambda$. The geometric picture of $\mathcal{CH}(X)$ does not depend on $\lambda$, the set of accessible models therefore does not depend on $\lambda$ either, and consequently the model selection inconsistency result holds for any possible choice of $\lambda$, algorithm to choose $\lambda$, or even oracle who says what $\lambda$ ought to be.

Our results show that model selection is impossible when $\kappa > \frac{1}{2}$; not only will the correct model not be chosen, the design actually makes it impossible to do so (with high probability). This nicely complements seminal results on model selection for the lasso, by \cite{zhao2006model}, which have the true model size $k$ growing as $O(n^c)$, with $c < 1$. In this setup, \cite{zhao2006model} showed that an almost necessary and sufficient condition for model selection consistency is that the ``irrepresentability conditions'' hold on the design matrix.

Another set of results, by \cite{wainwright2009sharp}, shows that model selection is impossible in the setup where the true model grows faster than $O(n/\log n)$. We consider our model selection inconsistency theorem a geometric description of these results. These results were later generalized in \cite{wainwright2007information} to get information-theoretic limits on sparsity recovery using a general decoder.

\bibliography{lasso_geometry}

\end{document}